\documentclass{amsart}
\usepackage{amsmath, amsthm}
\usepackage{amsmath,amscd}
\usepackage{amsfonts}
\usepackage{amssymb, latexsym}
\usepackage{xcolor}

\theoremstyle{plain}
\newtheorem{thm}{Theorem}[section]
\newtheorem{cor}[thm]{Corollary}
\newtheorem{prop}[thm]{Proposition}
\newtheorem{lem}[thm]{Lemma}

\theoremstyle{definition}
\newtheorem{defn}{Definition}[section]

\newcommand{\ga}{\alpha}
\newcommand{\gb}{\beta}

\newcommand{\gfi}{\varphi}

\newcommand{\gl}{\lambda}
\newcommand{\gp}{\pi}
\newcommand{\gm}{\mu}
\newcommand{\go}{\omega}
\newcommand{\gt}{\tau}

\newcommand{\gs}{\sigma}
\newcommand{\gz}{\zeta}
\newcommand{\gr}{\rho}

\newcommand{\N}{\mathbb{N}}
\newcommand{\Q}{\mathbb{Q}}
\newcommand{\C}{\mathcal{C}}
\newcommand{\F}{\mathcal{F}}
\newcommand{\FT}{\mathcal{FT}}
\newcommand{\FTP}{\mathcal{FT_{\gp}}}
\newcommand{\FTO}{\mathcal{FT_{\omega}}}
\newcommand{\Qx}{\mathbb{Q}[x]}
\newcommand{\nilR}{nilpotent $R$-powered group}
\newcommand{\nilQx}{nilpotent $\mathbb{Q}[x]$-powered group}
\newcommand{\NO}{\mathcal{N}_{\omega}}
\newcommand{\AO}{\mathcal{A}_{\omega}}
\newcommand{\GO}{\mathcal{G}_{\omega}}
\newcommand{\SO}{\mathcal{S}_{\omega}}

\copyrightinfo{2006} {American Mathematical Society}

\begin{document}

\title{Separability Properties of Nilpotent $\mathbb{Q}[x]$-Powered Groups II}

\author{Stephen Majewicz}
\address{Department of Mathematics\\
         CUNY-Kingsborough Community College\\
         Brooklyn, New York 11235}
\email{smajewicz@kbcc.cuny.edu}

\author{Marcos Zyman}
\address{Department of Mathematics\\
         CUNY-Borough of Manhattan Community College\\
         New York, New York 10007}
\email{mzyman@bmcc.cuny.edu}


\date{December 20, 2023}


\begin{abstract}
In this paper, we study nilpotent $\Qx$-powered groups that satisfy the following property: For some set of primes $\go$ in $\Qx,$ every $\go '$-isolated $\Qx$-subgroup in some family of its $\Qx$-subgroups is finite $\go$-type separable.
\end{abstract}

\maketitle

\section{Introduction}

This paper is a continuation of our paper entitled ``Separability Properties of Nilpotent $\Qx$-Powered Groups" \cite{majewicz_and_zyman-2020}.

\

Let $\C$ be a class of groups. A group $G$ is \emph{residually $\C$} if for every $1 \neq g \in G,$ there exists $K \in \C$ and a homomorphism $\gfi$ from $G$ onto $K$ such that $\gfi(g) \neq 1.$ A subgroup $H$ of a group $G$ is called \emph{$\C$-separable} in $G$ if, for every element $g \in G \setminus H,$ there is a homomorphism $\gfi$ of $G$ onto $K \in \C$ such that $\gfi(g) \notin \gfi(H).$ It is immediate that $G$ is residually $\C$ if and only if $\left\{ 1 \right\} $ is $\C$-separable in $G$.

Let $\mathcal{P}$ be a non-empty set of primes and $\mathcal{P} '$ the set of all primes not in $\mathcal{P}.$ A natural number $n$ is called a \emph{$\mathcal{P}$-number} if every prime divisor of $n$ belongs to $\mathcal{P}.$ By convention, $1$ is a $\mathcal{P}$-number for any set of primes $\mathcal{P}.$ A {\it $\mathcal{P}$-group} is a torsion group such that each of its elements has order a $\mathcal{P}$-number. Denote the class of all finite $\mathcal{P}$-groups by $\F _{\mathcal{P}},$ and let $G$ be any group with $H \leq G.$ If $H$ is $\F _{\mathcal{P}}$-separable in $G,$ then it can be shown that $H$ is a $\mathcal{P} '$-isolated subgroup of $G$ (i.e. if $h \in G$ and $h^{r} \in H$ for some $\mathcal{P} '$-number $r,$ then $h \in H$).

The converse, however, fails in general. For instance, in \cite{bardakov-2004}, Bardakov showed that if $G$ is a non-abelian group and $H$ is an isolated (finitely generated) subgroup of $G,$ then $H$ is not separable in the class of nilpotent groups (hence, $H$ is not $\F _{\{p \}}$-separable in $G$ for all primes $p$).

In \cite{sokolov-2014}, Sokolov continued the study of $\F _{\mathcal{P}}$-separability of subgroups in nilpotent groups. A group $G$ is said to have \emph{property $\mathcal{G}_{\mathcal{P}}$} (for some family of its subgroups) if every $\mathcal{P} '$-isolated subgroup of $G$ belonging to the family is $\F _{\mathcal{P}}$-separable. Prior to Sokolov's work, it was shown by Loginova \cite{loginova-1999} that for every prime $p,$ all finitely generated nilpotent groups have property $\mathcal{G}_{\{p\}}.$ This result can be generalized for a set of primes rather than a single prime. Two of Sokolov's results are:

\begin{enumerate}
\item An abelian group $A$ has property $\mathcal{G}_{\mathcal{P}}$ if and only if $A$ satisfies the following condition: In every quotient group $B$ of $A,$ all $p$-primary components of the torsion subgroup of $B$ have finite exponent for all $p \in \mathcal{P}.$

\vspace{.1in}

\item Every $\mathcal{P}$-restricted nilpotent group has property $\mathcal{G}_{\mathcal{P}}.$

\vspace{.1in}

\noindent (An abelian group is \emph{$\mathcal{P}$-restricted} if in every quotient $B$ of $A,$ all $p$-primary components of the torsion subgroup of $B$ are finite for all $p \in \mathcal{P}.$ A nilpotent group is \emph{$\mathcal{P}$-restricted} if it has a finite central series with $\mathcal{P}$-restricted abelian factors.)

\vspace{.1in}

\item A torsion-free nilpotent group has property $\mathcal{G}_{\mathcal{P}}$ if and only if it is $\mathcal{P}$-restricted.
\end{enumerate}

In this paper, we prove analogues of these results for nilpotent $\Qx$-powered groups. The ring $\Qx$ is an example of a \emph{binomial ring}, an integral domain with unity $R$ that has characteristic zero such that for any $r \in R$ and positive integer $k,$
\[
\binom{r}{k} = \frac{r(r - 1)\cdots(r - k + 1)}{k!} \in R.
\]

A {\emph{nilpotent $R$-powered group} is a nilpotent group $G$ that comes equipped with an action by $R$ such that $g^{\ga}$ is an element of $G$ for all $g \in G$ and $\ga \in R$, and satisfies the following:}
\begin{enumerate}
\item $g^{1} = g, \, g^{\ga}g^{\gb} = g^{\ga + \gb},$ and $(g^{\ga})^{\gb} = g^{\ga\gb}$ for all $g \in G$ and $\ga, \gb \in
R.$

\vspace{.1in}

\item $(h^{-1}gh)^{\ga} = h^{-1}g^{\ga}h$ for all $g, \ h \in G$ and for all $\ga \in R.$

\vspace{.1in}

\item If $\{g_{1}, \, \ldots, \, g_{n}\} \subset G$ and $\ga \in R,$ then
\[
g_{1}^{\ga} \cdots g_{n}^{\ga} = \tau_{1}(\bar{g})^{\ga}\gt_{2}(\bar{g})^{\binom{\ga}{2}} \cdots \gt_{k - 1}(\bar{g})^{\binom{\ga}{k-1}}\gt_{k}(\bar{g})^{\binom{\ga}{k}},
\]
where $k$ is the nilpotency class of the group generated by $\{g_{1}, \, g_{2}, \, \ldots, \, g_{n}\},$ $\bar{g} = (g_{1},
\, \ldots, \, g_{n}),$ and $\tau_{i}(\bar{g})$ is the $i$th \emph{Hall-Petresco word}.
\end{enumerate}

\vspace{.1in}

An \emph{abelian $R$-group} is a \nilR \, of class $1.$ If $G$ is an abelian $R$-group and $g, \, h \in G,$ then $(gh)^{\ga} = g^{\ga}h^{\ga}$ for any $\ga \in R.$ This follows from (3), together with the fact that $\tau_{1}(g, \, h) = gh.$ Thus, every abelian $R$-group can be viewed as an $R$-module. Throughout the paper, we use multiplicative and exponential notation for the operations in an abelian $R$-group.

Nilpotent $R$-powered groups naturally arise from taking the $R$-completion of a finitely generated torsion-free nilpotent group with respect to a given Mal'cev basis. Notions such as $R$-subgroup, $R$-morphism, direct products, $R$-series, etc. are defined in the obvious way (see \cite{hall-1969}, \cite{majewicz-2006}, and \cite{warfield}).

\vspace{.1in}

{\textbf{In this paper, all direct products will be restricted.}}

\vspace{.1in}

All necessary background material will be provided in the preliminary section. Throughout the paper, $R$ will always be a binomial ring (possibly with other qualifications). Unless otherwise told, we use ``$1$" to denote the unity element in a ring and the identity element in a group.

\vspace{.1in}

The definitions of residually $\C$ and $\C$-separable carry over to the category of \nilR s in a natural way.

\begin{defn}\label{d:Residual and R-Separability}
Let $\C$ be a class of \nilR s, and let $G$ be any \nilR.
\begin{itemize}
\item $G$ is \emph{residually $\C$} if for every $1 \neq g \in G,$ there exists $K \in \C$ and an $R$-homomorphism $\gfi$ from $G$ onto $K$ such that $\gfi(g) \neq 1.$ Equivalently, there exists a normal $R$-subgroup $N_{g}$ of $G$ with $g \notin N_{g}$ such that $G/N_{g} \in \C.$

\vspace{.1in}

\item An $R$-subgroup $H$ of $G$ is called \emph{$\C$-separable} in $G$ if, for every element $g \in G \setminus H,$ there is an $R$-homomorphism $\gfi$ of $G$ onto $K \in \C$ such that $\gfi(g) \notin \gfi(H).$ Equivalently, for every element $g \in G \setminus H,$ there exists a normal $R$-subgroup $N$ of $G$ such that $G/N \in \C$ and $g \notin HN.$
\end{itemize}
\end{defn}

\vspace{.1in}

We introduce some notation. Let $R$ be a UFD. {\textbf{We will assume throughout the paper that all sets of primes in $R$ are non-empty and no two elements in a given set are associates.}} If $\go$ is such a set, then $\go '$ will denote any set of primes satisfying the following:
\begin{itemize}
\item No element in $\go$ has an associate in $\go '.$
\item Every prime in $R$ has exactly one associate in either $\go$ or $\go '.$
\end{itemize}

\vspace{.1in}

For a set of primes $\go,$ we denote the class of \nilR s of finite $\go$-type by $\FTO$ (see Definition~\ref{d:finite0-type} below). The following definitions mimic those in Sokolov's paper:

\vspace{.02in}

\begin{itemize}
\item A \nilR \ $G$ has \emph{property $\GO$} for some family of its $R$-subgroups if every $\go '$-isolated $R$-subgroup in this family is $\FTO$-separable.

\vspace{.1in}

\item An abelian $R$-group $G$ satisfies \emph{Condition ($B_{\go}$)} if for every $R$-quotient $G/N$ of $G$ and every $\gp \in \go,$ the $\gp$-primary component of $G/N$ is bounded.

\vspace{.1in}

\item An abelian $R$-group $G$ is called \emph{$\go$-restricted} if for every $R$-quotient $G/N$ of $G$ and every $\gp \in \go,$ the $\gp$-primary component of $G/N$ is finitely $R$-generated (and, thus, finite $\go$-type).

\vspace{.1in}

\item A \nilR \ is \emph{$\go$-restricted} if it has a finite central $R$-series with $\go$-restricted abelian $R$-factors.
\end{itemize}

\vspace{.1in}

Our focus will be on the case $R = \Qx.$ Let $\go$ be a set of primes in $\Qx.$ The main results in this paper are:

\vspace{.05in}

\begin{itemize}
\item A necessary and sufficient condition for an abelian $\Qx$-group to have property $\GO$ is that is satisfies Condition ($B_{\go}$).

\vspace{.1in}

\item Every $\go$-restricted \nilQx \ has property $\GO$ for its family of normal $\Qx$-subgroups.

\vspace{.1in}

\item If $G$ is a $\Qx$-torsion-free \nilQx \ and has property $\GO,$ then each $\Qx$-factor group $\gz_{i + 1}G/\gz_{i}G$ of its upper central $\Qx$-series has property $\GO.$
\end{itemize}

\section{Preliminaries}

In this section, we provide the terminology and results on \nilR s which will be needed in this paper. Let $G$ be a \nilR. The following appear in \cite{clement_majewicz_zyman-2017}, \cite{hall-1969}, and \cite{warfield}:
\begin{itemize}
\item If $H$ is an $R$-subgroup (normal $R$-subgroup) of $G,$ then we write $H \leq_{R} G$ ($H \unlhd_{R} G$). We say that $H$ is \emph{$R$-generated by $X \subseteq G$} if $H$ is the smallest $R$-subgroup of $G$ containing $X.$ In this situation, we write $H = gp_{R}(X).$ If $X$ is finite, then we say that $H$ is \emph{finitely $R$-generated}.

\vspace{.1in}

\noindent \underline{Notation}: If $\ga \in R,$ then $G^{\ga} = gp_{R}(g^{\ga} \ | \ g \in G).$ Note that $G^{\ga} \unlhd_{R} G.$

\vspace{.1in}

\item If $N \unlhd_{R} G,$ then the $R$-action on $G$ induces an $R$-action on $G/N,$
\[
(gN)^{\ga} = g^{\ga}N \hbox{ for all } gN \in G/N \hbox{ and } \ga \in R,
\]
which turns $G/N$ into a \nilR.

\vspace{.1in}

\item If $H \leq_{R} G$ and $N \unlhd_{R} G,$ then $HN \leq_{R} G$ and $HN = gp_{R}(H, \ N),$ the smallest $R$-subgroup containing $H$ and $N$.
Furthermore, the kernel of an $R$-homomorphism $\gfi : G_{1} \rightarrow G_{2}$ is a normal $R$-subgroup of $G_{1}.$ It follows that the usual isomorphism theorems carry over to nilpotent $R$-powered groups. We will refer to them as the \emph{$R$-isomorphism theorems}. If $G_{1}$ and $G_{2}$ are $R$-isomorphic \nilR s, then we write $G_{1} \cong_{R} G_{2}.$

\vspace{.1in}

\item All upper central subgroups of $G,$ written as $\gz_{i}G,$ are $R$-subgroups of $G.$ The center of $G$ will be denoted by $Z(G).$
\end{itemize}

\vspace{.2in}

The concept of a $P$-number for a set of primes $P \subset \N$ is generalized in the next definition which is a slight improvement of our earlier definition in \cite{majewicz_and_zyman-2012(2)}. We remind the reader that all sets of primes in $R$ are non-empty and no pair of elements in such sets are associates.

\begin{defn}\label{d:o-element}
Let $R$ be a UFD and $\go$ a set of primes in $R.$ A non-zero element $\ga \in R$ is an \emph{$\go$-member} if $\ga$ is a non-unit and each prime divisor of $\ga$ has an associate in $\go.$
\end{defn}

\begin{defn}\cite{majewicz_and_zyman-2012(2)}\label{d:TORSION}
Let $R$ be a UFD and $\go$ a set of primes in $R.$ Suppose that $G$ is a \nilR. An element $g \in G$ is an \emph{$\go$-torsion element} if $g^{\ga} = 1$ for some $\go$-member $\ga.$ The set of $\go$-torsion elements of $G$ is written as $\gt_{\go}(G).$ If every element of $G$ is $\go$-torsion, then $G$ is an \emph{$\go$-torsion group}. We say that $G$ is \emph{$\go$-torsion-free} if the only $\go$-torsion element of $G$ is $1.$
\end{defn}

If $\gp$ is a prime in $R$ and $\go = \{\gp\},$ then we use the terms $\gp$-torsion and $\gp$-torsion-free and write $\gt_{\gp}(G)$ for the set of $\gp$-torsion elements of $G.$ It is clear that if $\go_{1}$ and $\go_{2}$ are sets of primes in $R,$ then $\gt_{\go_{1}}(G) = \gt_{\go_{2}}(G)$ if every element in $\go_{1}$ has an associate in $\go_{2}$ and vice versa. In particular, $\gt_{\gp_{1}}(G) = \gt_{\gp_{2}}(G)$ whenever $\gp_{1}$ and $\gp_{2}$ are prime associates.

\vspace{.1in}

The next definition does not require $R$ to be a UFD.

\begin{defn}\cite{warfield}\label{d:RTorsion}
Let $G$ be a \nilR.
\begin{itemize}
\item An element $g \in G$ is an \emph{$R$-torsion element} if $g^{\ga} = 1$ for some $0 \neq \ga \in R.$ The set of $R$-torsion elements of $G$ is denoted by $\gt(G).$

\item If every element of $G$ is $R$-torsion, then $G$ is called an \emph{$R$-torsion group}. We say that $G$ is \emph{$R$-torsion-free} if the only $R$-torsion element of $G$ is $1.$
\end{itemize}
\end{defn}

\begin{thm}\cite{majewicz_and_zyman-2012(2)}, \cite{warfield}\label{t:TorsionIsNormal}
Let $G$ be a \nilR.
\begin{itemize}
\item If $R$ is a UFD and $\go$ is a set of primes in $R,$ then $\gt_{\go}(G) \unlhd_{R} G$ and $G/\gt_{\go}(G)$ is $\go$-torsion-free.

\item If $R$ is any binomial ring, then $\gt(G) \unlhd_{R} G$ and $G/\gt(G)$ is $R$-torsion-free.
\end{itemize}
\end{thm}

\begin{defn}\cite{majewicz_and_zyman-2012(2)}\label{d:finite0-type}
Let $R$ be a UFD and $\go$ a set of primes in $R.$ A finitely $R$-generated $\go$-torsion group $G$ is said to be of \emph{finite $\go$-type}. If $\go = \{ \gp \}$ for a prime $\gp \in R,$ then $G$ is of \emph{finite $\gp$-type}.
\end{defn}

The classes of \nilR s of finite $\go$-type and finite $\gp$-type will be denoted by $\FTO$ and $\FTP$ respectively.

\vspace{.2in}

The next definition does not require $R$ to be a UFD.

\begin{defn}\cite{majewicz-2006}\label{d:finitetype}
A \nilR \ is of \emph{finite type} if it is a finitely $R$-generated $R$-torsion group.
\end{defn}

The class of \nilR s of finite type will be denoted by $\FT.$

\vspace{.15in}

\begin{thm}\cite{kargapolov and et al-1969}\label{t:MAX}
If $R$ is a noetherian domain, then every $R$-subgroup of a finitely $R$-generated \nilR \ is finitely $R$-generated.
\end{thm}

\begin{cor}\cite{majewicz_and_zyman-2009}\label{c:MAX}
Let $G$ be a \nilR, where $R$ is a noetherian domain, and $\go$ a set of primes in $R.$
\begin{itemize}
\item If $G \in \FT$ and $H \leq_{R} G,$ then $H \in \FT.$ If $H \unlhd_{R}G,$ then $G/H \in \FT.$

\vspace{.05in}

\item Suppose that $R$ is also a UFD. If $G \in \FTO$ and $H \leq_{R} G,$ then $H \in \FTO.$ If $H \unlhd_{R}G,$ then $G/H \in \FTO.$
\end{itemize}
\end{cor}

It is well-known that the torsion subgroup of a nilpotent group is a direct product of its $p$-torsion subgroups, where $p$ ranges over the primes (see \cite{clement_majewicz_zyman-2017}). Similarly, we have:

\begin{thm}\cite{majewicz-2006}\label{t:DecompositionTheorem}
Let $R$ be a PID and $G$ a \nilR. Let $\mathbb{P}$ be a set of primes in $R$ such that each prime in $R$ has exactly one associate in $\mathbb{P},$ and let $\go$ be a set of primes in $R.$
\begin{itemize}
\item If $G$ is an $R$-torsion group, then $G = \prod_{\gp \in \mathbb{P}}\gt_{\gp}(G).$

\vspace{.05in}

\item If $G$ is an $\go$-torsion group, then $G = \prod_{\gp \in \go}\gt_{\gp}(G).$
\end{itemize}
\end{thm}

\vspace{.1in}

Next we define the order ideal of an element belonging to a \nilR \ and the exponent ideal of a \nilR. Some of this appears in \cite{majewicz_and_zyman-2009} and mimics the definitions from module theory.

\begin{defn}\label{d:ORDER}
Let $R$ be a PID and $G$ a \nilR \, with $g \in G.$ The set
\[
Ann(g) = \{ \ga \in R \ | \ g^{\ga} = 1\}
\]
is an ideal of $R$ called the \emph{order ideal} or \emph{annihilator} of $g.$ We say that $g$ has \emph{infinite (or unbounded) order} whenever $Ann(g) = \{0\}.$
\end{defn}

Suppose that $Ann(g) \neq \{0\}.$ Since $R$ is a PID, there exists an element $0 \neq \gb \in R$ such that $Ann(g) = \, <\gb>,$ where $<\gb>$ is the ideal generated by $\gb.$

\begin{defn}\label{d:Order}
The element $\gb$ is called the \emph{non-infinite (or bounded) order} of $g.$
\end{defn}

Clearly, any associate of $\gb$ also generates $Ann(g).$ Hence, the order of $g$ is unique up to a unit factor.

\begin{defn}\label{d:ExponentOfNilR}
Let $R$ be a PID. The \emph{exponent ideal} of a \nilR \ $G$ is
\[
Ann(G) = \{ \ga \in R \ | \ g^{\ga} = 1 \hbox{ for all } g \in G\}.
\]

\vspace{.05in}

If $Ann(G) = \{0\},$ then $G$ has \emph{infinite (or unbounded) exponent}. In this case, $G$ is said to be \emph{unbounded}. Otherwise, $G$ is termed \emph{bounded}.
\end{defn}

As before, there exists an element $0 \neq \gl \in R$ such that $Ann(G) = \, <\gl>$ whenever $Ann(G) \neq \{0\}.$

\begin{defn}\label{d:EXPONENT}
The element $\gl$ is the \emph{non-infinite (or bounded) exponent} of $G.$
\end{defn}

The bounded exponent $\gl$ of $G$ is unique up to a unit factor since any associate of $\gl$ also generates $Ann(G).$

\vspace{.1in}

Every bounded \nilR \ is an $R$-torsion group. However, an $R$-torsion group is not necessarily bounded. Furthermore, if $G$ is a \nilR \ and $\ga \in R,$ then $G/G^{\ga}$ is bounded.

\vspace{.2in}

\noindent \underline{Examples:} Let $R$ be a PID.

\vspace{.05in}

\noindent 1) \ Let $\go$ be a set of primes in $R.$ Each element of a \nilR \ of finite $\go$-type has bounded order and the group is bounded (see \cite{majewicz_and_zyman-2009}). The same is true for nilpotent $R$-powered groups of finite type.

\vspace{.1in}

\noindent 2) \ Let $G$ be a \nilR\ and $\gp$ a prime in $R.$
\begin{itemize}
\item The direct product
\[
G/G^{\gp} \times G/G^{\gp^{2}} \times G/G^{\gp^{3}} \times \cdots
\]
is an unbounded $\gp$-torsion group.

\vspace{.1in}

\item The direct product
\[
G/G^{\gp} \times G/G^{\gp} \times G/G^{\gp} \times \cdots
\]
is a bounded $\gp$-torsion group.
\end{itemize}

\vspace{.1in}

\noindent 3) \ Let $\gp$ be a prime in $R.$ Suppose that $G$ is an abelian $R$-group with $R$-generating set $\{g_{1}, \ g_{2}, \ g_{3}, \ \ldots \}$ with $g_{1} \neq 1$ and subject to the relations
\[
g_{1}^{\gp} = 1, \ g_{2}^{\gp} = g_{1}, \ g_{3}^{\gp} = g_{2}, \ \ldots.
\]
Then $G$ is an unbounded $\gp$-torsion group called the \emph{$\gp$-quasicyclic $R$-group}. It will be denoted by $Q(\gp^{\infty}).$

\vspace{.2in}

The next theorem is well-known for modules (see \cite{bourbaki-1958}).

\begin{thm}\label{t:Bounded Module}
Let $R$ be a PID. Every bounded abelian $R$-group is a direct product of cyclic $R$-groups.
\end{thm}

\vspace{.1in}

An important result due to N. Blackburn states that if $g$ and $h$ are elements in a nilpotent group and $p$ is a prime, then for any $n \in \N$ the product  $g^{p^{n}}h^{p^{n}}$ is a $p^{m}$th power of some element of the group and some $m$ (see \cite{clement_majewicz_zyman-2017} or \cite{warfield} for details). A related result is:

\begin{thm}\cite{majewicz_and_zyman-2010}\label{t:divisible1}
Suppose $R$ contains $\Q,$ and let $G$ be a \nilR. If $\gb \in R$ and $g_{1}, \ \ldots, \ g_{m} \in G,$ then there exists $h \in G$
such that
\[
g_{1}^{\gb} \cdots g_{m}^{\gb} = h^{\gb}.
\]
\end{thm}

\vspace{.075in}

Hence, every element in $G^{\gb}$ is a $\gb$th power of an element of $G.$

\vspace{.1in}

\begin{defn}\label{d:Divisible Module}
Let $G$ be an abelian $R$-group.
\begin{itemize}
\item $G$ is called \emph{$R$-divisible} if for every $g \in G$ and $0 \neq \ga \in R,$ there exists $h \in G$ such that $g = h^{\ga}.$ We say that $h$ is an \emph{$\ga$th root} of $g.$

\vspace{.05in}

\item If $0 \neq \gb \in R,$ then $G$ is \emph{$\gb$-divisible} if every element of $G$ has a $\gb$th root.
\end{itemize}
\end{defn}

Any $R$-quotient of an $R$-divisible group is $R$-divisible whenever $R$ is a domain. Furthermore, the direct product of a collection of $R$-divisible ($\gb$-divisible) groups is $R$-divisible ($\gb$-divisible).

\vspace{.1in}

It is easy to show that the $p$-quasicyclic group is divisible when $p$ is an ordinary prime (see Lemma 5.25 in \cite{clement_majewicz_zyman-2017}). The same proof gives:

\begin{lem} \label{l:QuasiDivisible}
Let $R$ be a PID and $\pi$ a prime in $R$. The $\pi$-quasicyclic $R$-group is $R$-divisible.
\end{lem}

\vspace{.13in}

The following theorem is fundamental in the theory of modules.

\begin{thm}\cite{bourbaki-1958}\label{t:Divisible Module2}
Every abelian $R$-group can be embedded in an $R$-divisible group.
\end{thm}

\vspace{.13in}

The next few results involve abelian $R$-groups that are residually $\FT.$

\begin{lem}\label{l:Residually Finite P-Type 1}
Let $R$ be a PID. Every cyclic $R$-group is residually $\FT.$
\end{lem}

\begin{proof}
Suppose that $G = gp_{R}(g)$ is a cyclic $R$-group. There are 2 cases to consider.

\vspace{.1in}

\noindent \underline{Case 1}: Assume $g$ has infinite order. Choose $\ga, \ \gb \in R$ such that $g^{\ga} \notin gp_{R}\left(g^{\gb}\right),$ and set $N = gp_{R}\left(g^{\gb}\right).$ If $g^{\gl} \in G$ for some $\gl \in R,$ then $g^{\gl}N \in G/N$ is an $R$-torsion element since
\[
\left(g^{\gl}N\right)^{\gb} =  \left(g^{\gb}N\right)^{\gl} = N.
\]
And so, $G/N \in \FT.$  Furthermore, $g^{\ga}$ is not in the kernel of the natural map $G \rightarrow G/N.$ Therefore, $G$ is residually $\FT.$

\vspace{.1in}

\noindent \underline{Case 2}: Suppose that $g$ has bounded order and choose $0 \neq \gm \in R$ such that $Ann(g) = \, <\gm>.$ Since $G \cong_{R} R/Ann(g),$ we have that
\[
G \cong_{R} \frac{gp_{R}(g)}{gp_{R}\left(g^{\gm}\right)}.
\]
Clearly, $G \in \FT.$
\end{proof}

\vspace{.1in}

If $\go$ is a set of primes in a PID $R,$ then $\go$-torsion cyclic $R$-groups are clearly residually $\FTO.$ Furthermore, if $\mathcal{Q}$ is a property of \nilR s for any $R$ and $G_{1}, G_{2}, \ \ldots$ are residually $\mathcal{Q},$ then the direct product of the $G_{i}$ is also residually $\mathcal{Q}.$  This, together with Lemma~\ref{l:Residually Finite P-Type 1}, gives:

\begin{cor}\label{c:Residually Finite P-Type 1}
Let $R$ be a PID.
\begin{enumerate}
\item A direct product of cyclic $R$-groups is residually $\FT.$

\item If $\go$ is a set of primes in $R,$ then a direct product of $\go$-torsion cyclic $R$-groups is residually $\FTO.$
\end{enumerate}
\end{cor}

\begin{thm}\label{t:Malcev}
Let $R$ be a PID. An abelian $R$-group $G$ is residually $\FT$ if and only if it does not contain non-trivial elements with $\gb$th roots for all $\gb \in R.$
\end{thm}

\begin{proof}
Suppose that $g \in G$ ($g \neq 1$) has no $\ga$th root for some $\ga \in R,$ so that $g \notin G^{\ga}.$ Clearly, $G/G^{\ga}$ is bounded. By Theorem~\ref{t:Bounded Module}, $G/G^{\ga}$ is a direct product of cyclic $R$-subgroups. Hence, $G/G^{\ga}$ is residually $\FT$ by Corollary~\ref{c:Residually Finite P-Type 1}. Therefore, there exists a normal $R$-subgroup $N/G^{\ga}$ of $G/G^{\ga}$ such that $(G/G^{\ga})/(N/G^{\ga})$ is finite type and $g \notin N.$ And so, $G$ is residually $\FT.$ The converse follows immediately from the next lemma.
\end{proof}

\begin{lem}
If $G$ is a non-trivial $R$-divisible group and $H \in \FT,$ then the only $R$-homomorphism from $G$ to $H$ is the trivial one.
\end{lem}

\begin{proof}
Since $H$ is finite type, it has bounded exponent. Let $\gl$ be the exponent of $H,$ and let $g \in G.$ Since $G$ is $R$-divisible, there exists $h \in G$ such that $g = h^{\gl}.$ If $\gfi : G \rightarrow H$ is an $R$-homomorphism, then
\[
\gfi(g) = \gfi\left(h^{\gl}\right) = (\gfi(h))^{\gl} = 1.
\]
And so, $\gfi$ must be the trivial $R$-homomorphism.
\end{proof}

\begin{defn}\cite{majewicz_and_zyman-2012}\label{d:go-isolated}
Let $R$ be a UFD and $\go$ a set of primes in $R.$ An $R$-subgroup $H$ of a \nilR \, $G$ is \emph{$\go$-isolated} in $G$ if $g \in G$ and $g^{\gl} \in H$ for some $\go$-member $\gl$ imply that $g \in H.$
\end{defn}

``$R$-isolated" is defined in the obvious way and does not require a UFD. The next result is obvious.

\begin{lem}\label{l:torsionfree and isolated}
Let $R$ be a UFD and $G$ a \nilR. If $H \unlhd_{R} G,$ then $H$ is $\go$-isolated in $G$ if and only if $G/H$ is $\go$-torsion-free.
\end{lem}

\vspace{.1in}

Yet another concept that we will need is the height of a group element.

\begin{defn}\label{d:Height}
Let $R$ be a PID and $\gp$ a prime in $R.$ Let $G$ be a $\gp$-torsion abelian $R$-group. An element $g \in G$ has \emph{height} $n$ if $g$ has a $\gp^{n}$th root in $G$ but no $\gp^{n + 1}$th root in $G.$ If $g$ has a $\gp^{k}$th root for all $k \in \N,$ then $g$ is said to have \emph{infinite height}.
\end{defn}

We need to know when an abelian $R$-group can be expressed as a direct product of cyclic $R$-groups. The next theorem provides two instances of when this can be done. It is a generalization of well-known results due to Pr\"{u}fer and Kulikov.

\begin{thm}\cite{bourbaki-1958}\label{t:Bounded1}
Let $R$ be a PID and $\gp$ a prime in $R.$
\begin{enumerate}
\item A $\gp$-torsion abelian $R$-group $G$ is a direct product of cyclic $R$-subgroups if and only if $G$ is the union of an ascending sequence
\[
G_{1} \subseteq G_{2} \subseteq \cdots \subseteq G_{n} \subseteq \cdots
\]
of $R$-subgroups, such that the heights in $G$ of the non-identity elements in $G_{n}$ are bounded.
\item A countably $R$-generated $\gp$-torsion abelian $R$-group without elements of infinite height is a direct product of cyclic $R$-subgroups.
\end{enumerate}
\end{thm}

\vspace{.1in}

We end this section with a structure theorem for certain $R$-divisible groups.

\begin{thm}\cite{bourbaki-1958}\label{t:Divisible Module1}
Let $R$ be a PID. Every $R$-divisible group is $R$-isomorphic to a direct product of copies of $Q(\gp^{\infty})$ for various primes $\gp \in R$ and copies of the field of fractions of $R.$
\end{thm}

\section{MAIN SECTION}

Let $\C$ be a class of \nilR s. Two relationships between $\C$-separability and $\C$-residual properties for \nilR s are captured in the next lemma. These are mentioned in \cite{sokolov-2014} for ordinary nilpotent groups. We include the proof for completeness.

\begin{lem}\label{l:Separability and Residual}
Let $G$ be a \nilR.
\begin{enumerate}
\item $G$ is residually $\C$ if and only if $\{1\}$ is $\C$-separable in $G.$
\item Let $H \unlhd_{R} G.$ If $R$-homomorphic images of groups in $\C$ are in $\C,$ then $G/H$ is residually $\C$ if and only if $H$ is $\C$-separable in $G.$
\end{enumerate}
\end{lem}

\begin{proof}
We only prove (2) since (1) is clear. Assume first that $G/H$ is residually $\C.$ Let $g \in G \setminus H$, so that $gH \neq H$ in $G/H$. Then there exists $K \in \C$ and an onto $R$-homomorphism
$ \widehat{\varphi}: G/H \rightarrow K$ such that $\widehat{\varphi} (gH) \neq 1$ in $K$. Let
$p: G \rightarrow G/H$ be the canonical $R$-homomorphism. Then $\varphi= \widehat{\varphi} \circ p$ is an $R$-homomorphism from $G$ onto $K$. By construction,
\[
\varphi(g) = \widehat{\varphi}(gH) \neq 1
\] in $K$. If $\varphi(g)=\varphi(h)$ for some $h \in H$ we would have $\varphi (g)=1$, which is impossible. This shows that $\varphi (g) \notin \varphi(H)$, and hence, $H$ is $\C$-separable in $G$.

For the converse, assume that $H$ is $\C$-separable in $G$ and let $xH \neq H$ in $G/H$. Then $x \notin H$ and there exists $\widehat{K} \in \C$ and an onto $R$-homomorphism
\[
\widehat{\varphi}: G \rightarrow \widehat{K}
\]
such that $\widehat{\varphi}(x) \notin \widehat{\varphi}(H)$. Consider the \nilR \, $K = \widehat{K} / \widehat{\varphi}(H)$. Let
\[
\varphi: G/H \rightarrow K
\]
be given by $\varphi(gH) = \widehat{\varphi}(g) \widehat{\varphi}(H) \in K$. A routine check confirms that $\varphi$ is a well-defined onto $R$-homomorphism. Since $\C$ is closed under $R$-homomorphic images, $K \in \C$. Finally, since $\widehat{\varphi}(x) \notin \widehat{\varphi}(H)$, we conclude that
 $\varphi(xH) = \widehat{\varphi}(x) \widehat{\varphi}(H) \neq 1$. This proves that $G/H$ is residually $\C$, as promised.
\end{proof}

\vspace{.1in}

{\textbf{For the rest of the paper, we restrict ourselves to the binomial ring $\Qx.$}

\vspace{.1in}

Every prime in $\Qx$ has an associate that is a monic prime. Therefore, without loss of generality, we work with monic prime polynomials. Throughout the remainder of the paper, $\mathbb{M}$ will denote the set of all monic primes in $\Qx.$ Our assumptions for subsets of primes given in the introduction can be stated thus:
\begin{itemize}
\item Every set of primes is a non-empty subset of $\mathbb{M}.$
\item If $\go$ is a set of primes, then every $\go$-member is monic and $\go '$ is the set of monic primes such that $\go \cap \go ' = \emptyset$ and $\go \cup \go ' = \mathbb{M}.$
\end{itemize}

\vspace{.15in}

\begin{lem}\label{l:Seperable Implies Isolated}
Let $G$ be a \nilQx. If $H \leq_{\Qx} G$ and $H$ is $\FTO$-separable in $G$ for some set of primes $\go$ in $\Qx,$ then $H$ is $\go '$-isolated in $G.$
\end{lem}

\begin{proof}
Suppose that $H$ is not $\go '$-isolated. Then there exists $g \in G \setminus H$  and $\ga$ an $\go '$-member such that $g^{\ga} \in H.$ By hypothesis, there is a \nilQx \, of finite $\go$-type $K$ and an $\Qx$-homomorphism $\gfi$ from $G$ onto $K$ such that $\gfi(g) \notin \gfi(H).$ Since $K$ is $\go$-torsion, there exists an $\go$-member $\gm$ such that $(\gfi(g))^{\gm} = \gfi(g^{\gm}) = 1.$ Now, $\ga$ and $\gm$ are relatively prime because $\ga$ is not an $\go$-member. Since $\Qx$ is a PID, there exist $r, \, s \in \Qx$ such that $r \ga + s \gm = 1.$ Noting that $g^{\ga} \in H$ implies that $g^{r \ga} \in H,$ we get
\[
\gfi(g) = (\gfi(g))^{r \ga + s \gm} = \gfi(g^{r \ga}) (\gfi(g^{\gm}))^{s} = \gfi(g^{r \ga}) \in \gfi(H),
\]
a contradiction. We conclude that $H$ is $\go '$-isolated.
\end{proof}

We would like to find conditions for \nilQx s to satisfy the converse of Lemma~\ref{l:Seperable Implies Isolated}. This motivates the next definition.

\begin{defn}\label{d:Property GO}
Let $\go$ be a set of primes in $\Qx.$ A \nilQx \ $G$ has \emph{property $\GO$} for some family of its $\Qx$-subgroups if every $\go '$-isolated $\Qx$-subgroup in this family is $\FTO$-separable.
\end{defn}

For simplicity, we will just write ``$G$ has property $\GO$" when the family of $\Qx$-subgroups of $G$ being referred to is the family of \textit{all of its $\Qx$-subgroups}.

\vspace{.2in}

We will be interested in those abelian $\Qx$-groups that satisfy the following:

\vspace{.1in}

\noindent {\textbf{\underline{Condition ($B_{\go}$)}}}: If $\go$ is a set of primes in $\Qx$ and $G$ is an abelian $\Qx$-group, then for every $\Qx$-quotient $G/N$ of $G$ and every $\gp \in \go,$ the $\gp$-primary component $\gt_{\gp}(G/N)$ of $G/N$ is bounded.

\begin{prop}\label{p:Condition Equivalence}
Let $\go$ be a set of primes in $\Qx.$ An abelian $\Qx$-group $G$ satisfies Condition ($B_{\go}$) if and only if none of the $\Qx$-quotients of $G$ contain $\gp$-quasicyclic $\Qx$-subgroups for every $\gp \in \go.$
\end{prop}

\begin{proof}
Assume that there is a factor $\Qx$-group $G/N$ and a prime $\gp \in \go$ such that $\gt_{\gp}(G/N)$ is unbounded. We assert that some $\Qx$-quotient of $G$ has a $\gp$-quasicyclic $\Qx$-subgroup.

Let $k_{i}N$ be a non-trivial element of $\gt_{\gp}(G/N)$ of order $\gp^{i}$ for some $i \geq 1$ (such an element exists since $\gt_{\gp}(G/N)$ is unbounded.)
Set $k_{\ell}N = (k_{i}N)^{\gp^{i - \ell}}$ and note that $k_{\ell}N$ has order $\gp^{\ell}$ for $1 \leq \ell \leq i - 1.$ The unboundedness of $\gt_{\gp}(G/N)$ allows us to find an element $k_{j}N$ of order $\gp^{j},$ where $j > i.$ Once again, if we put $k_{s}N = (k_{j}N)^{\gp^{j - s}},$ then $k_{s}N$ has order $\gp^{s}$ for $i + 1 \leq s \leq j - 1.$ We continue this process and obtain a countable set of (distinct) elements $k_{1}N, \ k_{2}N, \ \ldots$ of $\gt_{\gp}(G/N)$ such that the order of each $k_{i} N$ is $\gp^{i}$.

Put $K/N = gp_{\Qx}(k_{1}N, \ k_{2}N, \ \ldots),$ and observe that $K/N$ is a $\gp$-torsion group with a countably infinite set of $\Qx$-generators. There are two cases to consider, depending on whether or not $K/N$ contains an element of infinite height.

\vspace{.15in}

\noindent \underline{Case (i)}: Suppose that $K/N$ has no non-trivial elements of infinite height. By Theorem~\ref{t:Bounded1}, $K/N$ is a direct product of countably many cyclic $\Qx$-subgroups whose exponents need not be bounded. In this case, there is a surjective $\Qx$-homomorphism  $\gfi : K/N \rightarrow L$, where $L$ is a $\gp$-quasicyclic $\Qx$-group. By the $\Qx$-isomorphism theorems,
\[
L \cong_{\Qx} \frac{K/N}{P/N} \leq_{\Qx}\frac{G/N}{P/N} \cong_{\Qx} G/P,
\]
where $P/N$ is the kernel of $\gfi.$ Therefore, some $\Qx$-quotient of $G$ contains a $\gp$-quasicyclic $\Qx$-subgroup as asserted.

\vspace{.15in}

\noindent \underline{Case (ii)}: Assume that $aN$ is an element of $K/N$ of infinite height and $aN \neq N.$ By Theorems~\ref{t:Divisible Module2} and \ref{t:Divisible Module1}, there is a $\Qx$-monomorphism of $G/N$ into a divisible $\Qx$-group $D$ and $D$ is $\Qx$-isomorphic to a direct product of copies of $Q(\gs^{\infty})$ for various primes $\gs \in \Qx$ and copies of the field of fractions of $\Qx.$ This direct product must have a $\gp$-quasicyclic factor since
$aN$ has bounded order $\gp^{m}$ for some $m \in \N.$ Thus, there exists a $\Qx$-homomorphism $\gr$ of $D$ onto a $\gp$-quasicyclic factor (say $L$) whose kernel does not contain $aN.$

We show that there is a $\Qx$-homomorphism of $G$ onto $L.$ Regard $G/N$ as a $\Qx$-subgroup of $D$ and set $\psi = \gr | _{G/N}.$ We claim that $\psi$ has unbounded image. To begin with, note that there exist elements $x_{i} \in K$ ($i \in \N$) such that
\[
aN = (x_{1}N)^{\gp} = (x_{2}N)^{\gp^{2}} = \cdots = (x_{i}N)^{\gp^{i}} = \cdots =
\]
because $aN$ has infinite height. Since $aN \notin \ker \gr$,
\[
\gr(aN) = \psi(aN) = (\psi(x_{i}N))^{\gp^{i}} \neq 1
\]
for all $i \in \N$. Thus, the image of $\psi$ must be unbounded as claimed. Now, since the only non-trivial proper $\Qx$-subgroups of $L$ are bounded and $\psi(aN) \neq 1,$ the image of $\psi$ equals $L.$ And so, if $\gfi : G \rightarrow G/N$ is the natural $\Qx$-homomorphism, then $\psi \circ \gfi : G \rightarrow L$ is a surjective $\Qx$-homomorphism. It follows from the $\Qx$-isomorphism theorems that some $\Qx$-quotient of $G$ is $\Qx$-isomorphic to a $\gp$-quasicyclic $\Qx$-subgroup. This proves the assertion for Case (ii).

Next we prove by contradiction that if $G$ satisfies Condition ($B_{\go}$), then none of the $\Qx$-quotients of $G$ contain $\gp$-quasicyclic $\Qx$-subgroups for all $\gp \in \go.$ Suppose that this is not the case, and there is a $\Qx$-quotient $G/N$ containing a $\gp$-quasicyclic $\Qx$-subgroup, say $H/N$, for some $\pi \in \go$. Then $H/N \leq_{\Qx} \gt_{\gp}(G/N)$ since all $\gp$-quasicyclic $\Qx$-groups are $\gp$-torsion. Since $\gt_{\gp}(G/N)$ is bounded by hypothesis, so is $H/N.$ This contradicts the fact that $H/N$ is an unbounded $\gp$-torsion $\Qx$-group.
\end{proof}

\begin{thm}\label{t:Abelian Is Gomega}
Let $\go$ be a set of primes in $\Qx.$ An abelian $\Qx$-group $G$ has property $\GO$ if and only if it satisfies Condition ($B_{\go}$).
\end{thm}

\begin{proof}
Assume that $G$ does not satisfy Condition ($B_{\go}$).
By Proposition~\ref{p:Condition Equivalence}, there exists $N \leq_{\Qx} G$ such that $G/N$ contains a $\gp$-quasicyclic $\Qx$-subgroup for some $\gp \in \go.$ Set $\mathbb{P} = \mathbb{M} \setminus \{\gp\}.$ By Theorem~\ref{t:DecompositionTheorem}, the family $\left \{ \gt_{\gb}(G/N)  \right \}_{\gb \in \mathbb{P}}$ of $\Qx$-subgroups of the $\Qx$-torsion subgroup of $G/N$ is a direct product of its members. Put
\[
T/N = \prod_{\gb \in \mathbb{P}}\gt_{\gb}(G/N).
\]
If $\gfi : G/N \longrightarrow (G/N)/(T/N)$ is the natural $\Qx$-homomorphism and $L/N$ is a $\gp$-quasicyclic $\Qx$-subgroup of $G/N,$ then $\gfi(L/N)$ is also $\gp$-quasicyclic because $L/N$ has trivial intersection with $T/N$ =  $\ker \gfi.$ Thus, by the $\Qx$-isomorphism theorems, $G/T$ contains a $\gp$-quasicyclic $\Qx$-subgroup. Hence, $G/T$ contains a non-trivial $\Qx$-divisible $\Qx$-subgroup by Lemma~\ref{l:QuasiDivisible}. By Theorem~\ref{t:Malcev}, $G/T$ is not residually finite type and, consequently, not residually finite $\go$-type. By Lemma~\ref{l:Separability and Residual}, $T$ cannot be $\FTO$-separable in $G.$ On the other hand, $T$ is $\go '$-isolated in $G$ because $\gt(G/T)$ is a $\gp$-torsion group. To see why this is true, suppose that $g^{\ga} \in T$ for some $\ga \in \Qx$ that is not an $\go$-member. We claim that $g \in T.$ Well, $g^{\ga} \in T$ implies that $(gT)^{\ga} = T$ in $G/T.$ Since $\gt(G/T)$ is a $\gp$-torsion group, there exists $t \in \N$ such that $(gT)^{\gp^{t}} = T.$ Since $\ga$ is not an $\go$-member and $\gp \in \go,$ there exist elements $a, \ b \in \Qx$ such that $a \ga + b \gp^{t} = 1.$ Hence,
\[
gT = (gT)^{a \ga + b \gp^{t}} = (gT)^{a \ga} (gT)^{b \gp^{t}} = T.
\]
And so, $g \in T$ as claimed. Therefore, $G$ does not have property $\GO.$

Next suppose that $G$ satisfies Condition ($B_{\go}$), and let $H$ be an $\go '$-isolated $\Qx$-subgroup in $G.$ We claim that $H$ is $\FTO$-separable or, equivalently (by Lemma~\ref{l:Separability and Residual}), that $G/H$ is residually $\FTO.$ To simplify notation, we let $\overline{K}$ denote the $\Qx$-quotient $K/H$ for $H \leq_{\Qx} K$ and any $K \leq_{\Qx}G.$ Choose an element $aH \in \overline{G}$ such that $aH \neq H.$ We divide the proof into 3 cases which depend on the order of $aH.$

\vspace{.05in}

\noindent \underline{Case 1}: Suppose that $aH$ has monic prime power order, say $\gp^{r}$ for some monic prime $\gp \in \Qx$ and $r \in \N.$ Since $H$ is $\go '$-isolated in $G$ and $aH \neq H$ it must be that $\gp \in \go.$ (If, instead, $\gp \in \go ',$ then $a^{\gp^{r}} \in H$ would imply that $a \in H,$ a contradiction).

Let $\gp^{t}$ be the exponent of $\gt_{\gp}\left(\overline{G}\right)$ for some $t \in \N$ (the exponent is bounded because $G$ satisfies Condition ($B_{\go}$)). It is clear that $\overline{G}^{\gp^{t}} \cap \gt_{\gp}\left(\overline{G}\right) = \{ H \}.$ Hence, $aH \overline{G}^{\gp^{t}} \neq \overline{G}^{\gp^{t}}.$ Furthermore, the $\gp$-torsion group $\overline{G}/\overline{G}^{\gp^{t}}$ is bounded and, thus, $\Qx$-isomorphic to a direct product of $\gp$-torsion cyclic $\Qx$-subgroups by Theorem~\ref{t:Bounded Module}. And so, $\overline{G}/\overline{G}^{\gp^{t}}$ is residually $\FTP$ by Corollary~\ref{c:Residually Finite P-Type 1}. It follows that $\overline{G}$ is also residually $\FTP.$ Since $\pi \in \go,$ we conclude that $\overline{G}$ is residually $\FTO.$

\vspace{.1in}

\noindent \underline{Case 2}: Suppose that $aH$ has non-infinite order that is monic and not a power of a prime, say $\gp^{r} \gm \in \Qx,$ where $\gp \in \Qx$ is a monic prime, $r \in \N,$ and $\gm \in \Qx$ is monic and not divisible by $\gp.$ As in Case 1, we have that $\gp \in \go$; for otherwise, $a^{\gp^{r} \gm} \in H$ would imply that $a^{\gm} \in H$ since $H$ is $\go '$-isolated, and this would be impossible.

Set $\mathbb{P} = \mathbb{M} \setminus \{\gp\}.$ By Theorem~\ref{t:DecompositionTheorem} the $\Qx$-subgroups $\gt_{\gb}\left(\overline{G}\right)$ form a direct product for the various primes $\gb \in \mathbb{P}.$ Set
\[
N = \prod_{\gb \in \mathbb{P}}\gt_{\gb}\left(\overline{G}\right).
\]
We claim that $N$ is $\go '$-isolated in $\overline G$. To see this assume that $(gH)^{\ga} \in N$ for some $gH \in \overline G$ and $\go '$-member $\ga.$ Then $\gp$ is relatively prime to both $\ga$ and to the order of $(gH)^{\ga}$. Hence, $\gp$ is relatively prime to the order of $gH$ itself. This shows that $gH \in N$, and our claim holds. Note also that $(aH)^{\pi^r} \in N$ since the order of $(aH)^{\pi^r}$ is $\gm$, which is relatively prime to $\pi$.
Hence $(aH)N$ has prime power order in $\overline{G}/N$.

By Case 1, there exist $K \in \FTO$ and an onto $\Qx$-homomorphism
\[ \psi: \overline G / N \rightarrow K
\] such that $\psi [ (aH)N ] \neq 1$. If $\Phi$ is the natural $\Qx$-homomorphism from $\overline G$ onto $\overline G /N$, then $\Phi \circ \psi$ is a $\Qx$-homomorphism from $\overline G$ onto $K$ that maps $aH$ to a non-trivial element in $K$. This completes Case 2.

\vspace{.1in}

\noindent \underline{Case 3}: Let $aH$ have infinite order and set $C = gp_{\Qx}(aH).$ We assert that $\overline{G}/C$ satisfies Condition ($B_{\go}$). Well, if $\overline{K}/C \leq_{\Qx} \overline{G}/C,$ then
\[
\frac{\overline{G}/C}{\overline{K}/C} \cong_{\Qx} \overline{G}/\overline{K} \cong_{\Qx} G/K.
\]
Since $G$ satisfies Condition ($B_{\go}$), $\gt_{\gp}(G/K)$ is bounded for any $\gp \in \go.$ Hence $\gt_{\gp}\left(\left(\overline{G}/C\right)/\left(\overline{K}/C\right)\right)$ is also bounded and, consequently, $\overline{G}/C$ satisfies Condition ($B_{\go}$) as asserted. Thus, $\gt_{\gp}\left(\overline{G}/{C}\right)$ is bounded for each $\gp \in \go.$

Choose a specific prime $\ga \in \go.$ We prove that there exists a natural number $r$ such that $aH \notin \overline{G}^{\ga^{r}}.$ Suppose, on the contrary, that $aH \in \overline{G}^{\ga^{t}}$ for all $t \in \N.$ Then there exist
$g_{1}, \, g_{2}, \, \ldots \in G$ such that
\[
aH = \left( g_{1}H \right)^{\ga} = \left( g_{2}H \right)^{\ga^{2}} = \cdots
\]
in $\overline G$. So, in $\overline{G} / C$, we have
\[
C = (aH)C = \left( g_{1}H \right)^{\ga}C = \left( g_{2}H \right)^{\ga^{2}}C = \cdots .
\]
Consequently, $\left(g_{i} H\right) C \in \gt_{\ga}\left(\overline{G} / C \right)$ for all $i \in \N.$ We claim that the order of $(g_{i}H) C$ is precisely $\ga^{i}.$ This is obviously the case for $i = 1$. Let $i > 1$. Assume that the order of $(g_{i}H)C$ is $\ga^k$ for some $k < i$. Since $(g_{i}H)^{\ga^k} \in C = gp_{\Qx}(aH)$, there is a $\gm \in \Qx$ such that $(g_iH)^{\ga^k} = (aH)^{\gm}$. Then
\[
aH = (g_{i}H)^{\ga^i} = \left((g_{i}H)^{\ga^k} \right)^{\ga^{i - k}} = (aH)^{\gm \ga^{i - k}}.
\]
Hence, $ (aH)^{\gm \ga^{i - k} - 1} = H$. Since $aH$ is of infinite order, $\gm \ga^{i - k} - 1 = 0$, which is false. Thus, there exists $r \in \N$ such that
$aH \notin \overline G^{\ga^r}$, as claimed.

It is clear that $\overline G^{\ga^r}$ is $\go '$-isolated in $\overline G$, and that $(aH) \overline{G}^{\ga^{r}}$ has order a power of $\ga$ in $\overline G / \overline{G}^{\ga^r}$. Proceeding as in Case 2, we can find a group $L$ in $\FTO$ and a $\Qx$-homomorphism from $\overline G$ onto $L$ that maps $aH$ non-trivially. This completes the proof of the theorem.
\end{proof}

\begin{defn}\label{d:Abelian wrestricted}
Let $\go$ be a set of primes in $\Qx.$ An abelian $\Qx$-group $G$ is \emph{$\go$-restricted} if for every $\Qx$-quotient $G/N$ of $G$ and every $\gp \in \go$ the $\gp$-primary component $\gt_{\gp}(G/N)$ of $G/N$ is finitely $\Qx$-generated (hence, finite $\gp$-type).
\end{defn}

The class of $\go$-restricted abelian $\Qx$-groups will be denoted by $\AO.$ Note that every $\go$-restricted abelian $\Qx$-group satisfies Condition ($B_{\go}$) since \nilQx s of finite $\gp$-type are bounded.

\begin{lem}\label{l:FGAbel}
Let $\go$ be a set of primes in $\Qx.$ Every finitely $\Qx$-generated abelian $\Qx$-group is $\go$-restricted.
\end{lem}

\begin{proof}
If $G$ is a finitely $\Qx$-generated abelian $\Qx$-group and $N \leq_{\Qx}G,$ then $G/N$ is also finitely $\Qx$-generated abelian. By Theorem~\ref{t:MAX}, any $\Qx$-subgroup of $G/N$ is finitely $\Qx$-generated. In particular, $\gt_{\gp}(G/N)$ is finitely $\Qx$-generated.
\end{proof}

\begin{defn}\label{d:Nilpotent wrestricted}
Let $\go$ be a set of primes in $\Qx.$ A \nilQx \ is \emph{$\go$-restricted} if it has a finite central $\Qx$-series with $\go$-restricted abelian $\Qx$-factors.
\end{defn}

The class of $\go$-restricted \nilQx s will be denoted by $\NO.$ Clearly, $\AO \subset \NO.$

\begin{lem}\label{l:FGNilp}
Let $\go$ be a set of primes in $\Qx.$ Every finitely $\Qx$-generated \nilQx \ is $\go$-restricted.
\end{lem}

\begin{proof}
Let $G$ be a finitely $\Qx$-generated \nilQx. Suppose that
\[
\{1\} = G_{0} \leq_{\Qx} G_{1} \leq_{\Qx} \cdots \leq_{\Qx} G_{n} = G
\]
is any central $\Qx$-series. By Theorem~\ref{t:MAX}, each $G_{i}$ is finitely $\Qx$-generated. Thus, each $G_{i + 1}/G_{i}$ is finitely $\Qx$-generated abelian. The result follows from Lemma~\ref{l:FGAbel}.
\end{proof}

We record a useful lemma. Its proof is omitted.

\begin{lem}\label{l:Identities}
Let $G$ be a \nilR, and let $A, \ B,$ and $C$ be $R$-subgroups of $G.$ If $B \unlhd_{R} A,$ then $(B \cap C) \unlhd_{R} (A \cap C)$ and
\[
\frac{A \cap C}{B \cap C} \cong_{R} \frac{B(A \cap C)}{B}.
\]
If $C \unlhd_{R} G$ as well, then $BC \unlhd_{R} AC$ and
\[
\frac{AC}{BC} \cong_{R} \frac{A}{B(A \cap C)}.
\]
\end{lem}

\begin{prop}\label{p:Class Closures}
Let $\go$ be a set of primes in $\Qx.$ Every $\Qx$-subgroup and $\Qx$-quotient of an $\AO$-group ($\NO$-group) is an $\AO$-group ($\NO$-group).
\end{prop}

\begin{proof}
(i) \ Let $G \in \AO$ and $H \leq_{\Qx}G.$ Suppose $K \leq_{\Qx}H$ and choose $\gp \in \go.$ Clearly, $\gt_{\gp}(H/K) \leq_{\Qx} \gt_{\gp}(G/K).$ Since $\gt_{\gp}(G/K)$ is finitely $\Qx$-generated, so is $\gt_{\gp}(H/K)$ by Theorem~\ref{t:MAX}. Hence, $H \in \AO.$

Next let $N \leq_{\Qx} G$ and $M/N \leq_{\Qx} G/N.$ For any $\gp \in \go,$ $\gt_{\gp}(G/M)$ is finitely $\Qx$-generated because $G \in \AO.$ Since $G/M \cong_{\Qx} (G/N)/(M/N),$ we have $G/N \in \AO.$

\vspace{.1in}

\noindent (ii) \ Now suppose that $G \in \NO$ with a central $\Qx$-series
\[
\{1\} = G_{0} \leq_{\Qx} G_{1} \leq_{\Qx} \cdots \leq_{\Qx} G_{n} = G
\]
such that each $G_{i + 1}/G_{i}$ is $\go$-restricted abelian. If $H \leq_{\Qx} G,$ then
\[
\{1\} = G_{0} \cap H \leq_{\Qx} G_{1} \cap H \leq_{\Qx} \cdots \leq_{\Qx} G_{n} \cap H = H
\]
is a central $\Qx$-series for $H.$ By Lemma~\ref{l:Identities}, $(G_{i + 1} \cap H)/(G_{i} \cap H)$ is $\Qx$-isomorphic to a $\Qx$-subgroup of $G_{i + 1}/G_{i}.$ Since $G_{i + 1}/G_{i} \in \AO,$ so is $(G_{i + 1} \cap H)/(G_{i} \cap H)$ by (i) above. Hence, $H \in \NO.$

Next let $N \unlhd_{\Qx}G.$ The quotient $\Qx$-group $G/N$ has a central $\Qx$-series
\[
\{1\} = G_{0}N/N \leq_{\Qx} G_{1}N/N \leq_{\Qx} \cdots \leq_{\Qx} G_{n}N/N = G/N.
\]
By Lemma~\ref{l:Identities} and a $\Qx$-isomorphism theorem, $(G_{i + 1}N/N)/(G_{i}N/N)$ is $\Qx$-isomorphic to $(G_{i + 1}/G_{i})/(G_{i}(G_{i + 1} \cap N)/G_{i}).$ Hence, $(G_{i + 1}N/N)/(G_{i}N/N)$ is $\Qx$-isomorphic to a $\Qx$-quotient of $G_{i + 1}/G_{i}$ and, consequently, $G/N \in \NO$ by (i).
\end{proof}

\vspace{.2in}

\noindent \underline{Notation}: Let $\go$ be a set of primes in $\Qx.$ The class of \nilQx s having a finite subnormal $\Qx$-series with $\go$-restricted abelian $\Qx$-factors will be denoted by $\SO.$

\vspace{.1in}

Note that for a set of primes $\go$ of $\Qx,$ we have that $\NO \subset \SO$ since every central $\Qx$-series is a subnormal $\Qx$-series.

\begin{prop}\label{p:Class Closure 2}
Let $\go$ be a set of primes in $\Qx.$ Every $\Qx$-subgroup and $\Qx$-quotient of an $\SO$-group is an $\SO$-group.
\end{prop}

\begin{proof}
Replace ``central" with ``subnormal" in the proof of Proposition~\ref{p:Class Closures}.
\end{proof}

\begin{prop}\label{p:Abelian SOGroup is AOGroup}
Let $\go$ be a set of primes in $\Qx.$ An abelian $\Qx$-group is $\go$-restricted if and only if it is in $\SO.$
\end{prop}

\begin{proof}
Let $G$ be an abelian $\Qx$-group in $\SO,$ and let $H \leq_{\Qx} G.$ By Proposition~\ref{p:Class Closure 2}, $G/H \in \SO$ and, thus, $\gt_{\gb}(G/H) \in \SO$ for any $\gb \in \go.$ Fix a prime $\gp \in \go,$ and put $T = \gt_{\gp}(G/H).$ We claim that $T$ is finitely $\Qx$-generated, thus proving that $G$ is an $\AO$-group.

Since $T \in \SO,$ it has a $\Qx$-series
\begin{equation}\label{e:SUBNORMAL}
\{1\} = T_{0} \leq_{\Qx} T_{1} \leq_{\Qx} \cdots \leq_{\Qx} T_{k} = T
\end{equation}
with $\go$-restricted abelian $\Qx$-factors. Thus, each $\gt_{\gp}(T_{i + 1}/T_{i})$ is finitely $\Qx$-generated. Since each term in (\ref{e:SUBNORMAL}) is $\gp$-torsion, each $T_{i + 1}/T_{i}$ is $\gp$-torsion as well. Hence, $\gt_{\gp}(T_{i + 1}/T_{i}) = T_{i + 1}/T_{i}$ and thus, $T_{i + 1}/T_{i}$ is finitely $\Qx$-generated. And so, $T$ is finitely $\Qx$-generated as claimed. The converse is clear.
\end{proof}

\begin{cor}\label{c:Abelian SOGroup is AOGroup}
Let $\go$ be a set of primes in $\Qx.$ Every abelian $\Qx$-group in $\NO$ is $\go$-restricted.
\end{cor}

\begin{proof}
This follows from Proposition~\ref{p:Abelian SOGroup is AOGroup} and the fact that $\NO \subset \SO.$
\end{proof}

Our next goal is to prove that for any set of primes $\go$ in $\Qx,$ every $\go$-restricted \nilQx \ has property $\GO$ for its family of normal $\Qx$-subgroups. First we prove some preliminary results.

\begin{prop}\label{p:PREPARATORY PROPOSITION 1}
Let $\go$ be a set of primes in $\Qx$ and $G \in \NO.$ If $\gm$ is any $\go$-member, then all $\Qx$-factors of any central $\Qx$-series of $G/G^{\gm}$ are finite $\go$-type. Consequently, $G/G^{\gm}$ is finite $\go$-type.
\end{prop}

\begin{proof}
As usual, we may assume that all primes in $\go$ and all $\go$-members are monic. Let $\gm$ be an $\go$-member and denote the set of monic prime divisors of $\gm$ by $\gs.$ Clearly, $\gs$ is a finite set and $\gm$ is a $\gs$-member. Define the set $\gs '$ by $\gs ' = \go \setminus \gs.$ Suppose that
\[
\{1\} = G^{\gm}/G^{\gm} \leq_{\Qx} G_{1}/G^{\gm} \leq_{\Qx} \cdots \leq_{\Qx} G_{n}/G^{\gm} = G/G^{\gm}
\]
is an arbitrary central $\Qx$-series for $G/G^{\gm},$ and let $A = (G_{i + 1}/G^{\gm})/(G_{i}/G^{\gm})$ be any one of its $\Qx$-factors. We claim that $A$ is finite $\go$-type.

By Proposition~\ref{p:Class Closures}, $G/G^{\gm} \in \NO$ because $G \in \NO.$ It follows from Proposition~\ref{p:Class Closures} and Corollary~\ref{c:Abelian SOGroup is AOGroup} that $A \in \AO.$ Hence, $\gt_{\gp}(A)$ is finitely $\Qx$-generated for every $\gp \in \go.$ Now, $A$ is a $\gs$-torsion group because $G_{k}/G^{\gm}$ is $\gs$-torsion for every $k = 1, \ 2, \ \ldots, \ n.$

Thus, by Theorem~\ref{t:DecompositionTheorem},
\[
\gt_{\go}(A) = \prod_{\gp \in \gs}\gt_{\gp}(A) \times \prod_{\gp ' \in \gs '}\gt_{\gp '}(A) \cong_{\Qx} \prod_{\gp \in \gs}\gt_{\gp}(A).
\]
Since $\gs$ is finite, $\gt_{\go}(A)$ is finitely $\Qx$-generated. Therefore, $A$ is finite $\go$-type as claimed. Since there are finitely many terms in the central $\Qx$-series, $G/G^{\gm}$ is finite $\go$-type.
\end{proof}

\begin{prop}\label{p:PREPARATORY PROPOSITION 2}
Suppose that $\go$ is a set of primes in $\Qx$ and $G \in \NO.$ If $1 \neq g \in G$ has bounded order an $\go$-member, then $g \notin G^{\gb}$ for some $\go$-member $\gb.$
\end{prop}

\begin{proof}
Suppose that $g$ has bounded order an $\go$-member $\gm,$ and let $\gs$ be the set of monic prime divisors of $\gm.$ Note that $\gs$ is a finite set and $\gm$ is a $\gs$-member. Since $G \in \NO$ and $\gs \subseteq \go,$ there exists a central $\Qx$-series
\[
\{1\} = G_{0} \leq_{\Qx} G_{1} \leq_{\Qx} \cdots \leq_{\Qx} G_{n} = G
\]
of $G$ such that each $\Qx$-factor $\gt_{\gp}(G_{i + 1}/G_{i})$ is finitely $\Qx$-generated for every $\gp \in \gs.$ By Theorem~\ref{t:DecompositionTheorem} and the fact that $\gs$ is a finite set, we see that
\[
\prod_{\gp \in \gs}\gt_{\gp}(G_{i + 1}/G_{i}) = \gt_{\gs}(G_{i + 1}/G_{i})
\]
is finitely $\Qx$-generated. Abbreviate the exponent of a \nilQx \ by ``exp" and define
\[
\gb = \prod_{i = 0}^{n - 1} \hbox{exp}\left(\gt_{\gs}(G_{i + 1}/G_{i}))\right).
\]
We claim that $g \notin G^{\gb}.$ According to Theorem~\ref{t:divisible1}, every element of $G^{\gb}$ can be written as a power of $\gb.$ If it were true that $g \in G^{\gb},$ then there would exist $h \in G$ such that $h^{\gb} = g.$ Hence, $h^{\gb \gm} = g^{\gm} = 1.$ Now, $\gb \gm$ is a $\gs$-member since $\gb$ and $\gm$ are both $\gs$-members. Thus, $h^{\gb \gm} = 1$ implies that $h \in \gt_{\gs}(G).$ Consequently, $h^{\gb} = 1.$ Therefore, $g = 1,$ a contradiction.
\end{proof}

\begin{thm}\label{t:ORestricted Nilpotent Is Residually FO}
Let $\go$ be a set of primes in $\Qx$ and $G \in \NO.$ Then, every $\Qx$-torsion element of $G$ is $\go$-torsion if and only if $G$ is residually $\FTO.$
\end{thm}

\begin{proof}
Suppose $G$ is residually $\FTO,$ and let $g \neq 1$ be a $\Qx$-torsion element. There exists $H \in \FTO$ and a $\Qx$-homomorphism $\gfi$ of $G$ onto $H$ such that $\gfi(g) \neq 1.$ Since $g \in \gt(G),$ there exists $0 \neq \ga \in \Qx$ such that $g^{\ga} = 1.$ Hence, $1 = \gfi(g^{\ga}) = [\gfi(g)]^{\ga}.$ Since $\gfi(g) \neq 1$ and $H$ is $\go$-torsion, $\ga$ must be an $\go$-member. Therefore, $\gt(G)$ is $\go$-torsion.

Conversely, assume that $\gt(G)$ is an $\go$-torsion group. We prove that $G$ is residually $\FTO$ by induction on the class $c$ of $G.$
\begin{itemize}
\item If $G$ has class $1,$ then it is abelian. By assumption, $G$ is $\go '$-torsion free. Since $\gt(G) \cong_{\Qx} \gt(G)/\{1\},$ the trivial $\Qx$-subgroup of $G$ is $\go '$-isolated in $G$ by Lemma~\ref{l:torsionfree and isolated}. Now, $G$ is $\go$-restricted by Corollary~\ref{c:Abelian SOGroup is AOGroup}. Therefore, $G$ also satisfies Condition ($B_{\go}),$ and thus, has property $\GO$ by Theorem~\ref{t:Abelian Is Gomega}. Hence, $\{1\}$ is $\FTO$-separable in $G.$ By Lemma~\ref{l:Separability and Residual}, $G$ is residually $\FTO.$
\item Assume that $G$ has class $c > 1$ and that the result is true for all $\go$-restricted \nilR s of smaller class. Let $1 \neq g \in G.$ We claim that there exists a \nilQx \ of finite $\go$-type and a $\Qx$-homomorphism $\gfi$ of $G$ onto this group such that $\gfi(g) \neq 1.$
\begin{enumerate}
\item If $g \notin Z(G),$ then $gZ(G) \neq Z(G)$ in $G/Z(G).$ Now, $G/Z(G)$ is nilpotent of class less than $c.$ Furthermore, $\gt(G/Z(G))$ is an $\go$-torsion group because $\gt(G)$ is. By induction, there exists $L \in \FTO$ and a $\Qx$-homomorphism $\gfi$ from $G/Z(G)$ onto $L$ such that $\gfi(gZ(G)) \neq 1.$ Hence, if $\psi : G \rightarrow G/Z(G)$ is the natural $\Qx$-homomorphism, then $\gfi \circ \psi$ satisfies the claim.
\item Suppose $1 \neq g \in Z(G).$ Clearly, $\gt(Z(G))$ is $\go$-torsion since $\gt(G)$ is $\go$-torsion. By Proposition~\ref{p:Class Closures}, $Z(G) \in \NO.$ Hence, $Z(G) \in \AO$ by Corollary~\ref{c:Abelian SOGroup is AOGroup}. By the induction hypothesis, there is a $\Qx$-subgroup $N$ of $Z(G)$ such that $Z(G)/N \in \FTO$ and $g \notin N.$ Consequently, $gN \neq N$ and $(gN)^{\ga} = N$ in $G/N$ for some $\go$-member $\ga.$ Since $G/N \in \NO,$ there exists an $\go$-member $\gb$ such that $gN \notin (G/N)^{\gb}$ by Proposition~\ref{p:PREPARATORY PROPOSITION 2}. And so, $(G/N)/(G/N)^{\gb} \in \FTO$ by Proposition~\ref{p:PREPARATORY PROPOSITION 1}. If
\[
\gfi_{1} : G \rightarrow G/N \hbox{ \ and \ } \gfi_{2} : G/N \rightarrow \frac{G/N}{(G/N)^{\gb}}
\]
are the natural $\Qx$-homomorphisms, then $\gfi = \gfi_{2} \circ \gfi_{1}$ satisfies the claim.
\end{enumerate}
\end{itemize}
\end{proof}

We are now prepared to prove the following:

\begin{thm}\label{t:ORestricted Nilpotent HAS PROPERTY GO}
Let $\go$ be a set of primes in $\Qx.$ Every $\go$-restricted \nilQx \ has property $\GO$ for its family of normal $\Qx$-subgroups.
\end{thm}

\begin{proof}
Let $G \in \NO$ and assume that $H$ is an $\go '$-isolated normal $\Qx$-subgroup of $G.$ By Lemma~\ref{l:torsionfree and isolated}, $G/H$ is $\go '$-torsion-free. Since $G \in \NO,$ we have $G/H \in \NO$ by Proposition~\ref{p:Class Closures}. By Theorem~\ref{t:ORestricted Nilpotent Is Residually FO}, $G/H$ is residually $\FTO$ because $\gt(G/H)$ is $\go$-torsion. And so, $H$ is $\FTO$-separable in $G$ by Lemma~\ref{l:Separability and Residual}.
\end{proof}

\begin{thm}\label{t:UpperCentralFactors}
Let $\go$ be a set of primes in $\Qx.$ If $G$ is a $\Qx$-torsion-free \nilQx \ and has property $\GO,$ then each $\Qx$-factor group $\gz_{i + 1}G/\gz_{i}G$ of its upper central $\Qx$-series has property $\GO.$
\end{thm}

By Theorem~\ref{t:Abelian Is Gomega}, $\gz_{i + 1}G/\gz_{i}G$ also satisfies Condition ($B_{\go}$).

\begin{proof}
We begin by observing that $\gz_{i}G$ is $\Qx$-isolated in $G$ because $G$ is $\Qx$-torsion-free (see \cite{majewicz_and_zyman-2012}). If $H \unrhd_{\Qx} \gz_{i}G$ for some $H \leq_{\Qx} G,$ then it easily verified that $H$ is $\go '$-isolated in $G$ whenever $H/\gz_{i}G$ is an $\go '$-isolated $\Qx$-subgroup in $\gz_{i + 1}G/\gz_{i}G.$

We claim that $H/\gz_{i}G$ is $\FTO$-separable in $\gz_{i + 1}G/\gz_{i}G.$ By hypothesis, we know that $H$ is $\FTO$-separable in $G.$ Thus, for each $g \in \gz_{i + 1}G \setminus H,$ there exists a normal $\Qx$-subgroup $N$ of $G$ (which depends on $g$) such that $G/N \in \FTO$ and $g \notin HN.$
It is clear that
\[
\frac{N\gz_{i}G \cap \gz_{i + 1}G}{\gz_{i}G} \unlhd_{\Qx} \frac{\gz_{i + 1}G}{\gz_{i}G} \hbox{ \ \ and \ \ } g\gz_{i}G \notin \frac{H}{\gz_{i}G} \frac{(N\gz_{i}G \cap \gz_{i + 1}G)}{\gz_{i}G}.
\]
Using the $\Qx$-isomorphism theorems, we have
\begin{eqnarray*}
\frac{\gz_{i + 1}G/\gz_{i}G}{(N\gz_{i}G \cap \gz_{i + 1}G)/\gz_{i}G} & \cong_{\Qx} & \frac{\gz_{i + 1}G}{N\gz_{i}G \cap \gz_{i + 1}G} \cong_{\Qx} \frac{\gz_{i + 1}G(N\gz_{i}G)}{N\gz_{i}G}\\
                                                                     & \cong_{\Qx} & \frac{N\gz_{i + 1}G/N}{N\gz_{i}G/N} \leq_{\Qx} \frac{G/N}{N\gz_{i}G/N}.
\end{eqnarray*}
Thus, $(\gz_{i + 1}G/\gz_{i}G)/((N\gz_{i}G \cap \gz_{i + 1}G)/\gz_{i}G)$ is $\Qx$-isomorphic to a $\Qx$-subgroup of a $\Qx$-quotient of $G/N$ and, thus, finite $\go$-type by Corollary~\ref{c:MAX}. Since $g$ was chosen arbitrarily, $H/\gz_{i}G$ is $\FTO$-separable in $\gz_{i + 1}G/\gz_{i}G.$ This proves the claim and completes the proof.
\end{proof}

\vspace{.1in}

\noindent \underline{Open Problem:} Find other classes of nilpotent $\Qx$-powered groups for which every $\go$-restricted \nilQx \ has property $\GO$ for its family of $\Qx$-subgroups contained in each of these classes.

\end{document}